\documentclass{amsart}
\usepackage{amssymb}
\usepackage{graphicx}
\newcommand{\C}{\mathbb C}

\def\({\left(}
\def\){\right)}

\newtheorem{lema}{Lemma}[section]

\newtheorem*{teorema*}{Theorem}

\newtheorem{lemma}{Lemma}[section]

\newtheorem{theorem}[lema]{Theorem}

\newtheorem{proposition}[lema]{Proposition}

\newtheorem{definition}[lema]{Definition}

  {\par \hfill \fbox{}}
  {\par \hfill \fbox{}}
  {\par \hfill }
  {\par \hfill }

\def\beq{\begin{equation}}
\def\eeq{\end{equation}}

\def\epsilon{\varepsilon}

\begin{document}

\title{Expansivity property for Composition operators on Orlicz-Lorentz spaces }
\author{Romesh Kumar}
\address{Department of Mathematics, University of Jammu,
Jammu 180006, INDIA.} 
\email{romeshmath@gmail.com}

\author{Rajat Singh}
\address{Department of Mathematics, University of Jammu,
Jammu 180006, INDIA.} \email{rajat.singh.rs634@gmail.com}

\thanks{The second author is supported by UGC under the scheme of JRF}
\date{September 2022}

\begin{abstract} 
In this paper, we investigated the concept of expansivity for composition operators on Orlicz-Lorentz spaces.  We study necessary and sufficient conditions  for expansivity, positive expansivity and uniformly expansivity for composition operators $C_{\tau}$ on $\mathbb{L}^{\varphi,h}(\mu)$. We extend the of results of  \cite{MAI1}  into Orlicz-Lorentz space.The proofs of the  results are essentially based on the results of \cite{MAI1}
\end{abstract}

\maketitle

\section{Introduction and Preliminaries}
The dynamics of continuous linear operators on Banach spaces have received a lot of attention over the last three-four decades. Many links have been established between dynamical systems and other areas of mathematics, such as ergodic theory, number theory, and geometry of Banach spaces. Because of their versatility in the construction of examples in linear dynamics, operator theory, and its applications, many researchers have extensively studied the notion of expansivity and shadowing property in the context of a special class of operators, the weighted shifts. As a result, many dynamical properties for such operators have been thoroughly analyzed and characterized in recent decades, sometimes even before the property in question was fully understood in more general contexts. For more details one can see (\cite{B1}, \cite{NC13}, \cite{NC15}, \cite{NC2018}). \\
In \cite{NC2018}, they have investigated the relationship between expansivity, hypercyclicity, supercyclicity, Li-Yorke chaos, and shadowing. They gave the complete characterizations of weighted shifts that satisfy various notions of expansivity in the case where the Banach space is $C_{0}$ or $l_{p}(1\leq p<\infty)$. The concept of expansivity and shadowing property for composition operator on $L^{p}$-spaces $(1\leq p<\infty)$ has been investigated in \cite{MAI1}. In \cite{Bam2} and \cite{Raj} study the notion of expansivity and positive expansivity for Orlicz-spaces and Lorentz spaces respectively are. As the Orlicz-Lorentz spaces provide a common generalization of Orlicz and Lorentz spaces, it is natural to broaden the study of composition operators to a more general class.\\
Motivated by the results of \cite{MAI1}, in this paper we investigated the characterization of expansivity and positive expansivity for composition operators on Orlicz-Lorentz space. The paper is structured as follows: Section 1 is introductory and we cite certain definitions and results which will be used throughout this paper. In Section 2, we provide necessary and sufficient conditions for expansivity, positive expansivity and uniform expansivity of composition operators  in the setting of Orlicz-Lorentz spaces.\\
Now, we recall some basic facts about the Orlicz-Lorentz space which is useful throughout  this paper. For more details on Orlicz-Lorentz spaces  we refer to (\cite{BS}, \cite{ROL}, \cite{Rao}) and references therein.\\
Let $(X,\mathbb{A},\mu)$ be a measure space with positive measure and $\mathbb{L}^{0}$ represent the space of all equivalence measurable functions on $X$ which are identified as $\mu$-a.e.
We  now define the distribution function $\mu_{g}$ of $g\in L^{0}$ on $(0,\infty)$ as
$$\mu_{g}(\lambda)=\mu\{x \in X:|g(x)|>\lambda\},$$
and the non-increasing rearrangement of $g$ on $(0,\infty)$ is
$$g^{*}(t)=\inf\{\lambda>0:\mu_{g}(\lambda)\leq t\}=\sup\{\lambda>0:\mu_{g}(\lambda)>t\}.$$
A function $\varphi:[0,\infty) \to [0,\infty]$ is an Orlicz function if it is convex function with $\varphi(0)=0$ and $\varphi(s) \to \infty~\mbox{as}~s \to \infty$ such that $\varphi(s)<\infty$ for some $0<s<\infty$. Let $K=[0,\mu(X)]$. Let $h:K \to (0,\infty)$ be locally integrable and non-increasing function with respect to Lebesgue measure which is known as weight function. 
For a given $\varphi$ and $h$, the space
$$\mathbb{L}^{\varphi,h}(\mu)=\{ g\in \mathbb{L}^{0}:I_{\varphi,h}(\lambda g)<\infty~ \mbox{for some}~\lambda>0\},$$
where $$I_{\varphi,h}(\lambda g)=\int_{K}\varphi(\lambda g^{*}(t))h(t)dt.$$ is known as Orlicz-Lorentz space.
Now, we define the norm of $g$ as
 $$||g||_{\varphi,h}=\inf\{\lambda>0:I_{\varphi,h}(g/\lambda)\leq 1\}.$$ 
The space $\mathbb{L}^{\varphi,h}(\mu)$ with the above norm is a Banach space. Moreover, if $A\in \mathbb{A}$ with $0<\mu(A)<\infty$, then $||\chi_{A}||_{\varphi,h}=\frac{1}{\varphi^{-1}\(\frac{1}{\int_{0}^{\mu(\tau^{-n}(A))}h(t)dt}\)}$. If $w=1$, then $\mathbb{L}^{\varphi,h}$ represents the Orlicz space $\mathbb{L}^{\varphi}(\mu)$ and when $\varphi(t)=t$, then it represents the Lorentz space. $\mathbb{L}^{w}(\mu)$. 

First of all we state various growth conditions in $\varphi$, the Orlicz function $\varphi$ satisfies the $\Delta_{2}$-condition(for small u; for large u), if there is a positive constant $M$(a positive constant M and $s_{o}>0$ with $\varphi(s_{o})>0$; a positive constant M and $s_{o}>0$ with $\varphi(s_{o})<\infty$ ) such that
\begin{equation}\label{eqdelta}
	\varphi(2s)\leq M\varphi(s), ~\mbox{for all}~ s>0.
\end{equation}
We supposed that $\varphi$ is left continuous at $b_{\phi}$, where
$$b_{\phi}=\sup\{u>o:\phi(u)<\infty\}.$$
Also, define $$a_{\phi}=\inf\{u>0:\phi(u)>0\}.$$
Now, we recall the convergence of sequences in Orlicz-Lorentz spaces.
\begin{lemma}\label{lemma1}
Let $\{g_{n}\}_{n\geq 1}$ be a sequence from $\mathbb{L}^{\varphi,h}(\mu)$ and $g \in \mathbb{L}^{\varphi,h}(\mu)$. Then
\begin{itemize}
\item[(a)] If $||g_{n}-g||_{\varphi,h} \to 0$, then $I_{\varphi,h}(g_{n}) \to I_{\varphi,h}(g) $. converse holds only if $\varphi$ satisfies the $\Delta_{2}$ conditions for large $(s>0)$.
\item[(b)] If $\varphi$ satisfies the $\Delta_{2}$ conditions for large $(s>0)$ with $I_{\varphi,h}(g_{n}) \to I_{\varphi,h}(g) $ as $n \to \infty$ and $g_{n}\to g$ a.e., then $g_{n} \to g$ in norm.
\end{itemize}  
\end{lemma}

Throughout in this paper, we have $(X,\mathbb{A},\mu)$ be a measure space with measure positive. Let $\tau:X \to X$ be a measurable non-singular transformation i.e., $\tau^{-1}\in \mathbb{A},$ if $A \in \mathbb{A}$. A measurable non-singular transformation induces a bounded and continuous composition operator on $L^{\varphi,h}(\mu)$ if and only if there exist $M>0$ such that
\begin{equation}\label{eq1}
\mu \circ \tau^{-1}(A)\leq M \mu(A),~\forall~ A \in \mathbb{A}.
\end{equation} 

Since the equation(\ref{eq1}) holds for all $M\geq 1$. Then, we have 
$$(C_{\tau}g)^{*}(t)\leq g^{*}(t/M),~\mbox{for each}~t\geq 0.$$

\section{Expansivity for composition operator}
In this section, we first basic results from (NC expansivity) which will be useful throughout in this section.
\begin{proposition}\label{prop1}
Let $X$ be any Banach space and $T$ be an operator on $X$. Then
\begin{itemize}
\item[(a)] For each $0\neq x\in X,$ T is positively expansive if and only if $\displaystyle\sup_{n \in \mathbb{N}}||T^{n}x||=\infty.$
\item[(b)] T is uniformly positively expansive if and only if $\displaystyle\lim_{n \to \infty} ||T^{n}x||=\infty$, uniformly on $S_{X}$, where $S_{X}=\{x \in X:||x||=1\}$.\\
If $T$ is invertible, then
\item[(c)] For each $0\neq x\in X,$ T is expansive if and only if $\displaystyle\sup_{n \in \mathbb{Z}}||T^{n}x||=\infty.$
\item[(d)] T is uniformly expansive if and only if $S_{X}=A\cup B,$ where $\displaystyle\lim_{n \to \infty}||T^{n}x||=\infty$ uniformly on $A$ and $\displaystyle\lim_{n \to \infty}||T^{n}x||=\infty$, uniformly on $B$.
\end{itemize}
\end{proposition}
By setting $\mathbb{A}^{+}=\{A\in \mathbb{A}:0<\mu(A)<\infty\}$ and supposing that the composition operator $C_{\tau}$ on $L^{\varphi,h}(\mu)$, we rewrite the definition \cite{NC2018} and \cite{MAI1} for composition operators on Orlicz-Lorentz spaces $L^{\varphi,h}(\mu)$ in which $\varphi$ satisfies the $\Delta_{2}$ condition for large $s>0$.
\begin{definition}
The system $(X,\mathbb{A},\mu,\tau,\C_{\tau})$ is said to be a composition dynamical system if 
\begin{itemize}
\item[(a)] $(X,\mathbb{A},\mu)$ be a $\sigma$-finite measure space.
\item[(b)] $\tau:X\to X$ be bi-measurable transformation which is injective also.
\item[(c)] there is positive constant $M$ such that 
$$\mu(\tau^{-1}(A))\leq M\mu(A),~\forall~A\in \mathbb{A}.$$
\item[(d)] $C_{\tau}:L^{\varphi,h}(\mu) \to L^{\varphi,h}(\mu)$ is the composition operator induced by $\tau$ i.e.,
$$C_{\tau}g=g \circ \tau,~for~ g\in L^{\varphi,h}(\mu).$$
\end{itemize}
\end{definition} 
The condition (c) guarantees that the composition operator on $L^{\varphi,h}(\mu)$ is bounded. Moreover, if $\tau$ is bijective and $\tau^{-1}$ satisfies the condition (c) then $C_{\tau^{-1}}$ is also bounded linear operator and $C_{\tau}^{-1}=\C_{\tau^{-1}}.$ For complete details on composition operator interesting reader can see (rk singh).\\
AS it turns out from the literature for every $A\in \mathbb{A}$ with $0<\mu(A)<\infty$, we have $||\chi_{A}||_{p}=(\mu(A))^{\frac{1}{p}}$ and $||\chi_{A}||_{\varphi,h}=\frac{1}{\varphi^{-1}\(\frac{1}{\int_{0}^{\mu(A)}h(t)dt}\)}$, we can write the definition of \cite{MAI1} that are defined for $L^{p}$-space in the same context to Orlicz-Lorentz space. 
\begin{definition}
Let $(X,\mathbb{A},\mu)$ be a measure space and $\tau$ be a measurable non-singular transformation which is also invertible. Then the measurable set $W$ of $X$ is said to be wandering set for $\tau$ if $\{\tau^{-n}(w)\}_{n \in \mathbb{Z}}$ are pairwise disjoint.
\end{definition}




\begin{theorem}
Let $(X,\mathbb{A},\mu,\tau,C_{\tau})$ be composition dynamical system. Then
\begin{itemize}
\item[(a)] The operator $C_{\tau}$ is expansive if and only if for all $A \in \mathbb{A}$ with $0<\mu(A)<\infty$, we have 
$$\displaystyle\inf_{n \in \mathbb{Z}}\varphi^{-1}\(\frac{1}{\int_{0}^{\mu(\tau^{-n}(A))}h(t)dt}\)=0.$$
\item[(b)] The operator $C_{\tau}$ is positively expansive if and only if 
$$\displaystyle\inf_{n\in \mathbb{N}}\varphi^{-1}\(\frac{1}{\int_{0}^{\mu(\tau^{-n}(A))}h(t)dt}\)=0,$$
where $A\in \mathbb{A}$ and $0<\mu(A)<\infty$.
\item[(c)] If $\varphi$ satisfies the $\delta_{2}$ condition for large $s>0$, then the $C_{\tau}$ is uniformly positively expansive if and only if 
$$\displaystyle\lim_{n \to \infty}\frac{\varphi^{-1}\(\frac{1}{\int_{0}^{\mu(A)}h(t)dt}\)}{\varphi^{-1}\(\frac{1}{\int_{0}^{\mu(\tau^{-n}(A))}h(t)dt}\)}=\infty,$$ 
for all $A \in \mathbb{A}$ with $0<\mu(A)<\infty$.
\item[(d)] If $\varphi$ satisfies the $\Delta_{2}$ condition for $s>0$, then the composition operator $C_{\tau}$ is uniformly expansive if and only if $\mathbb{A}^{+}$ can be decomposed as $\mathbb{A}^{+}=\mathbb{A}_{B}^{+} \cup \mathbb{A}_{C}^{+}$ such that 
$$ \displaystyle\lim_{n \to \infty}\frac{\varphi^{-1}\(\frac{1}{\int_{0}{\mu(A)}h(t)dt}\)}{\varphi^{-1}\(\frac{1}{\int_{0}^{\mu(\tau^{n}(A))}h(t)dt}\)}=\infty,~\mbox{uniformly on}~\mathbb{A}_{B}^{+},$$ and
$$ \displaystyle\lim_{n \to \infty}\frac{\varphi^{-1}\(\frac{1}{\int_{0}{\mu(A)}h(t)dt}\)}{\varphi^{-1}\(\frac{1}{\int_{0}^{\mu(\tau^{-n}(A))}h(t)dt}\)}=\infty,~\mbox{uniformly on}~\mathbb{A}_{C}^{+}.$$
\end{itemize}
\end{theorem}

\begin{proof}
Necessity:- Suppose $C_{\tau}$ is expansive. Then by (c) of proposition(\ref{prop1}) 
$$\displaystyle\sup_{n\in \mathbb{Z}}||C_{\tau}^{n}g||_{\varphi,h}=\infty,~\forall~g \in L^{\varphi,h}(\mu)\setminus\{0\}.$$
Taking $g \in \chi_{A},$ where $A \in \mathbb{A}$ with $0<\mu(A)<\infty.$ Then
\begin{eqnarray*}
||C_{\tau}^{n}g||_{\varphi,h} &=& ||C_{\tau}^{n}\chi_{A}||_{\varphi,h}\\
                              &=& ||\chi_{\tau^{-n}(A)}||_{\varphi,h}\\
															&=& \frac{1}{\varphi^{-1}\(\frac{1}{\int_{0}^{\mu(\tau^{-n}(A))}h(t)dt}\)}.
\end{eqnarray*}
Thus, we have
\begin{eqnarray*}
\displaystyle\sup_{n \in \mathbb{Z}}||C_{\tau}^{n}g||_{\varphi,h} &=& \displaystyle\sup_{n \in \mathbb{Z}}\frac{1}{\varphi^{-1}\(\frac{1}{\int_{0}^{\mu(\tau^{-n}(A))}h(t)dt}\)}\\
                                                                 &=& \infty
\end{eqnarray*}
i.e., $$\displaystyle\lim_{n \in \mathbb{Z}}\varphi^{-1}\(\frac{1}{\int_{0}^{\mu(\tau^{-n}(A))}h(t)dt}\)=0.$$
Sufficiency:- Suppose 
\begin{equation}\label{equ1}
\displaystyle\lim_{n \in \mathbb{Z}}\varphi^{-1}\(\frac{1}{\int_{0}^{\mu(\tau^{-n}(A))}h(t)dt}\)=0
\end{equation}
for all $A\in \mathbb{A}$ with $0<\mu(A)<\infty.$\\
Let $g \in L^{\varphi,h}(\mu)\setminus \{0\}.$ Then there exist some $\epsilon>0$ such that the set $A^{\epsilon}=\{x \in X:|\tau(x)|>\epsilon\}$ has a positive measure. Since $(X,\mathbb{A},\mu)$ be $\sigma$-finite, so  we have $0<\mu(A^{\epsilon})<\infty.$ Now, for all $n \in \mathbb{Z}$ and using the fact that $h(Mu)\leq  M h(u),$ we have 
\begin{eqnarray*}
\int_{I}\varphi\(\frac{(\epsilon\chi_{A^{\epsilon}} \circ \tau^{n})^{*}(t)}{M||C_{\tau}^{n}g||_{\varphi,h}}\)h(t)dt &\leq& \int_{0}^{\mu(\tau^{-n}(A^{\epsilon}))}\varphi\(\frac{(g\circ \tau^{n})^{*}(t)}{M||C_{\tau}^{n}g||_{\varphi,h}}\)h(t)dt \\
                                                     &\leq& \int_{0}^{\mu(\tau^{-n}(A^{\epsilon}))}\varphi\(\frac{g^{*}(t\ M)}{M||C_{\tau}^{n}g||_{\varphi,h}} \)h(t)dt\\
																										&\leq& \int_{I}\varphi\(\frac{g^{*}(u)}{||C_{\tau}^{n}g||_{\varphi,h}}\)h(t)dt \\
                                                    &\leq& 1.
\end{eqnarray*}
That means, 
\begin{eqnarray*}
\epsilon.||\chi_{A} \circ \tau^{n}||_{\varphi,h} &\leq& M||C_{\tau}^{n}g||_{\varphi,h}\\
 \epsilon.\frac{1}{\varphi^{-1}\(\frac{1}{\int_{0}^{\mu(\tau^{-n}(A))}h(t)dt}\)}&\leq& M||C_{\tau}^{n}g||_{\varphi,h}\\
\end{eqnarray*}
and so it follows from eq(\ref{equ1}), $\displaystyle\sup_{n \in \mathbb{Z}}||C_{\tau}^{n}g||_{\varphi,h}=\infty.$
As $g$ is arbitrary, by using (c) of proposition (\ref{prop1}), we obtain that $C_{\tau}$ is expansive on $L^{\varphi,h}(\mu)$. \\
In order to proof the part (b), just replaces $\mathbb{Z}$ by $\mathbb{N}$ in the above proof. \\
(c) Suppose $C_{\tau}$ is uniformly positively expansive. Then by proposition (\ref{prop1}), we have 
\begin{equation}\label{eqa}
\displaystyle\lim_{n\to \infty} ||C_{\tau}^{n}g||_{\varphi,h}=\infty,~\mbox{uniformly on}~S_{L^{\varphi,\phi}(\mu)}.
\end{equation}
where $S_{L^{\varphi,\phi}(\mu)}=\{g\in L^{\varphi,\phi}(\mu):||g||_{\varphi,h}=1\}.$\\
By setting $g=\varphi^{-1}\(\frac{1}{\int_{0}^{\mu(A)}h(t)dt}\).\chi_(A),$ for all $A\in \mathbb{A}$ with $0<\mu(A)<\infty$.\\
Then $||g||_{\varphi,h}=1$ and also,
\begin{eqnarray*}
||C_{\tau}^{n}g||_{\varphi,h} &=& \varphi^{-1}\(\frac{1}{\int_{0}^{\mu(A)}h(t)dt}\)||C_{\tau}^{n}\chi_{A}||_{\varphi,h} \\
                             &=& \varphi^{-1}\(\frac{1}{\int_{0}^{\mu(A)}h(t)dt}\)||\chi_{\tau^{-n}(A)}||_{\varphi,h} \\
                             &=& \frac{\varphi^{-1}\(\frac{1}{\int_{0}^{\mu(A)}h(t)dt}\)}{\varphi^{-1}\(\frac{1}{\int_{0}^{\mu(\tau^{-n}(A)}h(t)dt}\)}.
\end{eqnarray*}
From eq(\ref{eqa}), we see that 
$$\displaystyle\lim_{n \to \infty} \frac{\varphi^{-1}\(\frac{1}{\int_{0}^{\mu(A)}h(t)dt}\)}{\varphi^{-1}\(\frac{1}{\int_{0}^{\mu(\tau^{-n}(A)}h(t)dt}\)}=\infty,~\mbox{uniformly on the sets}~A\in \mathbb{A}^{+}.$$
Conversely, as $\varphi$ satisfies the $\Delta_{2}$ condition for large $s>0$ and simple functions are dense in $L^{\varphi,\mu}$, so it is sufficient to show that for any simple function $g\in S_{L^{\varphi,h}(\mu)}$, $\displaystyle\lim_{n \to \infty}||C_{\tau}^{n}g||_{\varphi,h}=\infty.$\\
By hypothesis, for $M>0$, there exist $N_{c}\in \mathbb{N}$ such that $\forall~A\in \mathbb{A}^{+}$, we have
$$\frac{\varphi^{-1}\(\frac{1}{\int_{0}^{\mu(A)}h(t)dt}\)}{\varphi^{-1}\(\frac{1}{\int_{0}^{\mu(\tau^{-n}(A)}h(t)dt}\)}>M,~\forall~n \geq N_{c}.$$
Let $g=\sum_{i=1}^{m}\alpha_{i}\chi_{A_{i}}\in S_{L^{\varphi,h}(\mu)}$ be a simple function, where $\alpha_{i}\in \mathbb{C}\setminus \{0\}$, $A_{i}\in \mathbb{A}^{+}$ and suppose that $\{A_{i}\}_{i=1}^{m}$ are pairwise disjoint. In view of lemma\ref{lemma1}, for all n, we have
\begin{eqnarray*}
||C_{\tau}^{n}g||_{\varphi,h} &=& ||C_{\tau}^{n} \sum_{i=1}^{m}\alpha_{i}\chi_{A_{i}} \circ \tau^{n}||_{\varphi,h} \\
                              &\geq& a \sum_{i=1}^{m}|\alpha_{i}| ||\chi_{\tau^{-n}(A_{i})}||_{\varphi,h} \\
															&=& a\sum_{i=1}^{m}|\alpha_{i}| \frac{1}{\varphi^{-1}\frac{1}{\int_{0}^{\mu(\tau^{-n}(A_{i}))}h(t)dt}} \\
															&=& aM \sum_{i=1}^{m}|\alpha_{i}| \frac{1}{\varphi^{-1}\frac{1}{\int_{0}^{\mu(A_{i})}h(t)dt}}\\
															&\geq& aM \sum_{i=1}^{m}|\alpha_{i}| ||\chi_{A_{i}}||_{\varphi,h} \\
															&=& aM ||\sum_{i=1}^{m}\alpha_{i} \chi_{A_{i}}||_{\varphi,h} \\
															&=& aM ||g||_{\varphi,h}\\
															&=& aM.
\end{eqnarray*}
Thus, we get
$$\displaystyle\lim_{n \to \infty}||C_{\tau}^{n}g||_{\varphi,h}=\infty.$$
Writing $g=g^{+}-g^{-},$ where $g\in S_{L^{\varphi,h}(\mu)}$ and $g^{+},g^{-}$ are positive and negative measurable functions. Then a sequence of simple functions $\{g_{k}\}_{k\in \mathbb{N}}$ with $|g_{k}|\leq 2|g|$ such that converges point wise to g be found and so the sequence $\{\varphi(g_{k})\}$ also converges to $\{\varphi(g)\}$ and using the Lebesgue dominated converges theorem, we have 
$$\displaystyle\lim_{k\to \infty} I_{\varphi,h}(g_{k})=I_{\varphi,h}(g).$$
Thus, $\displaystyle\lim_{k\to \infty}||g_{k}||_{\varphi,h}=||g||_{\varphi,h}=1.$
It follows form above computation that $\displaystyle\lim_{k\to \infty}||C_{\tau}^{n}g||_{\varphi,h}=\infty$ and hence by proposition (\ref{prop1}) $C_{\tau}$ is uniformly expansive.\\
(d) Suppose that $C_{\tau}$ is uniformly expansive. Then by part (d) of proposition (\ref{prop1}), we have for $S_{L^{\varphi,h}(\mu)}=B\cup C$, where 
$$\displaystyle\lim_{n \to \infty}||C_{\tau}^{n}g||_{\varphi,h}=\infty~\mbox{uniformly on }~B$$ and $$\displaystyle\lim_{n \to \infty}||C_{\tau}^{-n}g||_{\varphi,h}=\infty~\mbox{uniformly on }~C.$$
By setting $\mathbb{A}^{+}=\{A \in \mathbb{A}:0<\mu(A)<\infty\}$ and using the fact that simple functions are dense in $L^{\varphi,h}(\mu)$, we have 
$\mathbb{A}^{+}=\mathbb{A}_{B}^{+} \cup \mathbb{A}_{C}^{+},$ where $\mathbb{A}_{B}^{+}=\{A\in \mathbb{A}:\varphi^{-1}\(\frac{1}{\int_{0}^{\mu(A)}h(t)dt}\)\chi_{A}\in B\}$ and $\mathbb{A}_{C}^{+}=\{A\in \mathbb{A}:\varphi^{-1}\(\frac{1}{\int_{0}^{\mu(A)}h(t)dt}\)\chi_{A}\in C\}.$
As $$\varphi^{-1}\(\frac{1}{\int_{0}^{\mu(A)}h(t)dt}\)\chi_{A}\in S_{L^{\varphi,h}(\mu)}$$ and $$C_{\tau}^{n}\(\varphi^{-1}\(\frac{1}{\int_{0}^{\mu(A)}h(t)dt}\)\chi_{A}\)=\varphi^{-1}\(\frac{1}{\int_{0}^{\mu(A)}h(t)dt}\)\chi_{\tau^{-n}(A)},$$ it follows that from our assumption, we get
$$\displaystyle\lim_{n \to \infty} \frac{\varphi^{-1}\(\frac{1}{\int_{0}{\mu(A)}h(t)dt}\)}{\varphi^{-1}\(\frac{1}{\int_{0}^{\mu(\tau^{n}(A))}h(t)dt}\)}=\infty,~\mbox{uniformly on}~\mathbb{A}_{B}^{+},$$ and $$\displaystyle\lim_{n\to \infty}\frac{\varphi^{-1}\(\frac{1}{\int_{0}{\mu(A)}h(t)dt}\)}{\varphi^{-1}\(\frac{1}{\int_{0}^{\mu(\tau^{-n}(A))}h(t)dt}\)}=\infty,~\mbox{uniformly on}~\mathbb{A}_{C}^{+}$$ as desired.
In order to prove the converse, in view of part (d) of proposition (\ref{prop1}) it is sufficient to prove that $S_{L^{\varphi,h}(\mu)}=B \cup C,$ where 
$$\displaystyle\lim_{n \to \infty}||C_{\tau}^{n}g||_{\varphi,h}=\infty, \mbox{uniformly on B}$$
and $$\displaystyle\lim_{n \to \infty}||C_{\tau}^{n}g||_{\varphi,h}=\infty, \mbox{uniformly on C}.$$
By given condition, let $M>0$. Then there exist $m \in \mathbb{N}$ such that, $\forall n\geq N$
$$\frac{\varphi^{-1}\(\frac{1}{\int_{0}{\mu(A)}h(t)dt}\)}{\varphi^{-1}\(\frac{1}{\int_{0}^{\mu(\tau^{n}(A))}h(t)dt}\)}>M,~\forall~A\in \mathbb{A}_{B}^{+}$$
and $$\frac{\varphi^{-1}\(\frac{1}{\int_{0}{\mu(A)}h(t)dt}\)}{\varphi^{-1}\(\frac{1}{\int_{0}^{\mu(\tau^{-n}(A))}h(t)dt}\)}>M,~\forall A\in \mathbb{A}_{C}^{+}.$$
Let $g=\sum_{i=1}^{m}\alpha_{i}\chi_{A_{i}} \in S_{L^{\varphi,h}(\mu)}$ be a simple function where $A_{i}'s$ are pairwise disjoint measurable sets with $\mu(A_{i})>0,~\forall~i=\{1,2,3,...m\}$. As $\mathbb{A}^{+}=\mathbb{A}_{B}^{+}\cup \mathbb{A}_{C}^{+},$ then \\
$g_{\mathbb{A}_{B}^{+}}=\sum_{i=1,A_{i}\in \mathbb{A}_{B}^{+}}^{m}\alpha_{i}\chi_{A_{i}}$ and $g_{\mathbb{A}_{C}^{+}}=\sum_{i=1,A_{i}\in \mathbb{A}_{C}^{+}}^{m}\alpha_{i}\chi_{A_{i}}.$
Clearly, $g=g_{\mathbb{A}_{B}^{+}}+g_{\mathbb{A}_{C}^{+}}$ and so by using lemma there is some $c>0$ such that 
$$||g||_{\varphi,h}\geq c(||g||_{\varphi,h}+||g||_{\varphi,h}) \geq c||g||_{\varphi,h}$$ and since $||g||_{\varphi,h}=1$, we have 
$$2\geq c(||g_{\mathbb{A}_{B}^{+}}||_{\varphi,h}+||g_{\mathbb{A}_{C}^{+}}||_{\varphi,h}) \geq 2c.$$
This implies that $$||g_{\mathbb{A}_{B}^{+}}||_{\varphi,h}\geq \frac{1}{2}~\mbox{or}~||g_{\mathbb{A}_{C}^{+}}||_{\varphi,h}\geq \frac{1}{2}.$$
now, if $||g_{\mathbb{A}_{B}^{+}}||_{\varphi,h}\geq \frac{1}{2}$, then we have for all $n\geq N_{c}$
\begin{eqnarray*}
||C_{\tau}^{-n}g||_{\varphi,h} &=& ||C_{\tau}^{-n}\sum_{i=1}^{m} \alpha_{i}\chi_{A_{i}}||_{\varphi,h} \\
                               &\geq& c \sum_{i=1}^{m} |\alpha_{i}| ||\chi_{\tau^{-n}(A_{i})}||_{\varphi,h}\\
															&=& c \sum_{i=1}^{m} |\alpha_{i}| \varphi^{-1}\frac{1}{\varphi^{-1}\(\frac{1}{\int_{0}^{\mu(\tau^{n}(A_{i}))}h(t)dt}\)} \\
															&>& c M \sum_{i=1}^{m} |\alpha_{i}| \varphi^{-1}\frac{1}{\varphi^{-1}\(\frac{1}{\int_{0}^{\mu(A_{i})}h(t)dt}\)} \\
															&=& cM \sum_{i=1}^{m} |\alpha_{i}| ||\chi_{A_{i}}||_{\varphi,h} \\
															&=& cM ||\sum_{i=1}^{m} \alpha_{i} \chi_{A_{i}}||_{\varphi,h} \\
                              &=& cM ||g_{\mathbb{A}_{B}^{+}}||_{\varphi,h}\\
															&\geq& \frac{c M}{2}.											
\end{eqnarray*}
Similarly, if $||g_{\mathbb{A}_{C}^{+}}||_{\varphi,h}\geq \frac{1}{2}$, then we have 
$$||C_{\tau}^{n}g||_{\varphi,h}\geq \frac{c M}{2}.$$
Put $S^{a}_{L^{\varphi,h}(\mu)}=\{ g\in S_{L^{\varphi,h}(\mu)}: \mbox{g is simple function} \}$. Then $S^{a}_{L^{\varphi,h}(\mu)}=B^{a}\cup C^{a}$, in which 
$B^{a}=\{ g \in S^{a}_{L^{\varphi,h}(\mu)}: ||g_{\mathbb{A}_{B}^{+}}||_{\varphi,h}\geq \frac{1}{2}\}$ and $C^{a}=\{ g \in S^{a}_{L^{\varphi,h}(\mu)}: ||g_{\mathbb{A}_{C}^{+}}||_{\varphi,h}\geq \frac{1}{2}\}.$\\
Since $g \in S_{L^{\varphi,h}(\mu)}$ is arbitrary and so the result holds for all simple function. Now, using the fact that simple functions are dense in $L^{\varphi,h}(\mu)$ and doing the same calculation as in part (C) of this Theorem and using the Lebesgue dominated convergence theorem, we get
$$\displaystyle\lim_{k \to \infty} I_{\varphi,h}(g_{k})=I_{\varphi,h}(g)$$ and so by lemma(\ref{lemma1}), we get
$$\displaystyle\lim_{k \to \infty} ||g_{k}||_{\varphi,h}=1.$$ So, there exist $k_{o}>0$ such that $\forall k\geq k_{o}$, $||g_{k}||_{\varphi,h}\geq \frac{1}{2}.$\\
Thus, \begin{eqnarray*}
||C_{\tau}^{n} \frac{g_{k}}{||g_{k}||_{\varphi,h}}||_{\varphi,h} &<& 2||C_{\tau}^{n}g_{k}||_{\varphi,h} \\
                                                                 &<& 2 \int_{I}\varphi(g_{k}\circ \tau)^{*}(t)h(t)dt \\
																																&<& 4 \int_{I}\varphi(g \circ \tau)^{*}(t)h(t)dt \\
\end{eqnarray*}
i.e., $$||C_{\tau}^{n} \frac{g_{k}}{||g_{k}||_{\varphi,h}}||_{\varphi,h}<4||C_{\tau}^{n}g||_{\varphi,h}$$ 
and at least one of the following set must be infinite either $F_{1}(g)=\{k\in \mathbb{N}:\frac{g_{k}}{||g_{k}||_{\varphi,h}} \in B^{a}\},~\mbox{or}~F_{2}(g)=\{k\in \mathbb{N}:\frac{g_{k}}{||g_{k}||_{\varphi,h}} \in C^{a}\}.$\\
Now, $F_{1}(g)$ is infinite, then there is sequence $\{k_{j}\}_{j \in \mathbb{N}},$ such that $\{\frac{g_{k_{j}}}{||g_{k_{j}}||_{\varphi,h}}\}_{j \in \mathbb{N}} \subset B^{a}$ and using the above calculation, we see that $$||C_{\tau}^{-n}\frac{g_{k_{j}}}{||g_{k_{j}}||_{\varphi,h}}||_{\varphi,h}>\frac{M}{2},~n\geq N and M>0$$ and thus, it follows that $$ ||C_{\tau}^{-n}g||_{\varphi,h}>\frac{M}{2^{3}},\forall~n\geq N.$$
Similarly, if $F_{2}(g)$ is infinite, then we get 
$$ ||C_{\tau}^{-n}g||_{\varphi,h}>\frac{M}{2^{3}},\forall~n\geq N.$$
Taking $B=\{g \in S_{L^{\varphi,h}(\mu)}:F_{1}(g)=\infty\}$ and $C=\{S_{L^{\varphi,h}(\mu)}:F_{2}(g)=\infty\}$, then $S_{L^{\varphi,h}(\mu)}=B\cup C,$ such that 
$\displaystyle\lim_{n \to \infty}||C_{\tau}^{n}g||_{\varphi,h},$ uniformly on B and $\displaystyle\lim_{n \to \infty}||C_{\tau}^{-n}g||_{\varphi,h},$ uniformly on C. 

\end{proof}

\end{document}